\documentclass{amsart}
\usepackage{amsthm,amsmath,amsthm,amssymb,mathabx,mathrsfs,graphicx,mathtools,caption,relsize,hyperref,cleveref,enumerate,xcolor}

\theoremstyle{plain}
 \newtheorem{thm}{Theorem}[section]
 \newtheorem{prop}[thm]{Proposition}
 \newtheorem{lem}[thm]{Lemma}
 \newtheorem{cor}[thm]{Corollary}
\theoremstyle{definition}
 
 \newtheorem{dfn}[thm]{Definition}
\theoremstyle{remark}
 
 \numberwithin{equation}{section}
 
\renewcommand{\leq}{\leqslant}
\renewcommand{\geq}{\geqslant}
\renewcommand{\setminus}{\smallsetminus}
\def\GL{\mbox{\rm GL}}
\def\PGL{\mbox{\rm PGL}}
\def\stab{\text{\rm Stab}}
\def\Aut{\text{\rm Aut}}
\def\Fix{\text{\rm Fix}}
\def\id{\text{\rm id}}
\def\con{\text{\sf con}}
\def\Q{\mathbf{a}}

\title[Decomposition Theorems for Automorphism Groups of Trees]{Decomposition Theorems for Automorphism Groups of Trees}
\date{\today}


\author[Max Carter]{Max Carter} 
\address{School of Mathematical and Physical Sciences \\ 
University of Newcastle \\ 
Callaghan \\
Australia}
\email{max.carter@newcastle.edu.au}

\author[George Willis]{George Willis}
\address{School of Mathematical and Physical Sciences \\ 
University of Newcastle \\ 
Callaghan \\
Australia}
\email{george.willis@newcastle.edu.au}

\thanks{This research was supported by the Australian Research Council grant FL170100032.} 

\begin{document}

\begin{abstract}
Motivated by the Bruhat and Cartan decompositions of general linear groups over local fields, double cosets of the group of label preserving automorphisms of a label-regular tree over the fixator of an end of the tree and over maximal compact open subgroups are enumerated.  This enumeration is used to show that every continuous homomorphism from the automorphism group of a label-regular tree has closed range.
\end{abstract}

\maketitle

\section{Introduction}

Totally disconnected locally compact groups (t.d.l.c.\ groups here on after) are a broad class which includes Lie groups over local fields. It is not known, however, just how broad this class is, and one approach to investigating this question is to attempt to extend ideas about Lie groups to more general t.d.l.c.\ groups. This approach has (at least partially) motivated works such as \cite{Loc,Glo1,Glo2}.

That is the approach taken here for Bruhat and Cartan decompositions of general linear groups over local fields, which are respectively double-coset enumerations of $\GL(n,\mathbb{Q}_p)$ over the subgroups, $B$, of upper triangular matrices, and $\GL(n,\mathbb{Z}_p)$, of matrices with $p$-adic integer entries. Such decompositions of real Lie groups are given in \cite[\S\S VII.3 and 4]{Knapp}, and of Lie groups over local fields in \cite{BenOh,Bru,DelSec,IwaMat}. Automorphism groups of regular trees have features in common with the groups $\GL(2,\mathbb{Q}_p)$ (in particular, $\PGL(2,\mathbb{Q}_p)$ has a faithful and transitive action on a regular tree, \cite[Chapter II.1]{Serre}). Here, we find double coset enumerations of automorphism groups of trees that correspond to Bruhat and Cartan decompositions. 

The pioneering work on automorphism groups of trees is the 1970 paper \cite{Ref1} by Jacques Tits. In his paper, Tits shows that the automorphism groups of regular trees and of trees with a regular vertex labelling in most cases have a non-trivial simple subgroup, and that the tree and vertex labelling can be recovered from this subgroup. In the present note, we identify subgroups of Tits' groups that correspond to groups of upper triangular matrices and of matrices over $p$-adic integers and enumerate the double cosets over these subgroups. The main difference is that, in the case of the Bruhat decomposition, the number of double cosets is no longer finite and, in the case of the Cartan decomposition, double coset representatives are no longer powers of a single element. 

The Bruhat and Cartan decompositions have further implications for the structure of Lie groups, both real and $p$-adic, and their unitary representations. It may be shown, for example, that continuous homomorphisms from simple groups in these classes have closed range, and we use the generalised Cartan decomposition to show that the same holds for automorphism groups of label-regular trees. In the special case when the double coset representatives are powers of a single element, the contraction subgroup for that element is an essential ingredient in the proof, and a key step in our argument is to derive a corresponding method for the Cartan decomposition derived here. 

Examples of simple t.d.l.c.\ groups not all of whose homomorphic images are closed may be found in Le Boudec's articles \cite{Ref3, Ref4}. These groups act on regular trees and are not closed as subgroups of the full automorphism group. The property of having all homomorphic images closed or, more fundamentally, the generalised contraction group property for double coset representatives, thus sheds light on the question of how typical $p$-adic Lie groups are among general t.d.l.c. groups. 


\section{Label-Regular Trees and their Automorphism Groups}\label{sec:prelim}

Label-regular trees and their automorphism groups were first studied in \cite{Ref1}, although not with that name.  Terminology and results that will be used here are recalled in this section and versions of the Bruhat and Cartan decompositions for these groups given in the next section. The vertex and edge sets of the tree $\mathcal{T}$ will be denoted by $V\mathcal{T}$ and $E\mathcal{T}$ respectively. All trees considered here are \emph{locally finite}, {\it i.e.\/}, all vertices are incident on only finitely many edges. 

A \emph{path} in $\mathcal{T}$ is a sequence of vertices $( v_i )_{i \in I} \subseteq V\Gamma$, with $I$ an indexing set, $v_i$ adjacent $v_{i+1}$ for all $i \in I$, and no vertex in $V\Gamma$ occurring more than once. The path $( v_i )_{i \in I} \subseteq V\Gamma$ is a \emph{ray} if $I = \mathbb{N}$ and is \emph{bi-infinite} if $I = \mathbb{Z}$.  An \emph{end} of $\mathcal{T}$ is an equivalence class of rays, with two rays being equivalent if their intersection is also a ray. The set of all ends of $\mathcal{T}$ is called the \emph{boundary} and is denoted $\partial \mathcal{T}$. The distance between two vertices $u,v \in V\mathcal{T}$ is the number of edges on the shortest path between $u$ and $v$. 

Let $\mathcal{T}$ be an infinite tree with all vertices having degree greater than~$1$. A \emph{labelling} of $\mathcal{T}$ is a map $\lambda: V\mathcal{T} \rightarrow C$ with $C$ set of \emph{labels}. For $v \in V\mathcal{T}$, let $N(v)$ be the set of all vertices adjacent to $v$ in $\mathcal{T}$ and define a multiset $L(v) = \{ \lambda(w) : w \in N(v) \}$. Then $\mathcal{T}$ is a \emph{label-regular tree} if $L(v)$ depends only on the value of $\lambda(v)$, that is, $v_1,v_2 \in V\mathcal{T}$ having the same label implies that they have the same numbers of neighbours with each label. Following~\cite[\S5.2]{Ref1}, the labelling is \emph{normal} if $\lambda$ is surjective and the group of automorphisms preserving the labelling (see below) is transitive on the sets $\lambda^{-1}(i)$, $i\in C$. All labellings here are assumed to be normal.

Denote, for a label-regular tree $\mathcal{T}$ with normal labelling, the number of vertices with the label $j\in C$ neighbouring each vertex with label $i\in C$ by $a_{ij}$. Then $\Q = (a_{ij})_{i,j\in C}$ is a $C\times C$ matrix that determines  $\mathcal{T}$ up to isomorphism. The matrix $\Q$ need not be symmetric but does satisfy that $a_{ij}=0\iff a_{ji}=0$. Further the graph $G_\Q$, that has $C$ as its vertex set and $\{i,j\}$ is an edge if $a_{ij}\ne0$, is connected, see~\cite[Proposition~5.3]{Ref1}. Conversely, for every $C\times C$ matrix $\Q = (a_{ij})_{i,j\in C}$ with non-negative integer entries and such that $a_{ij}=0 \iff a_{ji}=0$ and graph $G_\Q$ connected, there is a label-regular tree with labels from $C$ such that every vertex labelled $i$ has $a_{ij}$ neighbours labelled $j$. This tree, which will be denoted by $\mathcal{T}_\Q$, is unique up to isomorphism and is called the \emph{$\Q$-covering of $C$} in~\cite{Ref1}. 

The group of automorphisms, $\Aut(\mathcal{T}_\Q)$, of $\mathcal{T}_\Q$ is the group of all automorphisms of the underlying tree that also preserve the labelling, {\it i.e.\/}, of all $\varphi\in \Aut(\mathcal{T}_\Q)$ satisfying $\lambda(\varphi(v)) = \lambda(v)$ for all $v \in V\mathcal{T}_\Q$. For $G\leq \Aut(\mathcal{T}_\Q)$, the \emph{stabiliser} subgroup of $Y\subseteq V\mathcal{T}_\Q$, denoted $\stab_{G}(Y)$, is the set of all $g \in G$ satisfying $gY = Y$. Similarly, the \emph{fixator}, denoted $\Fix_{G}(Y)$, of $Y$ is the subgroup of all $g \in G$ satisfying $gy = y$ for all $y \in Y$. We shall be interested in closed subgroups of $\Aut(\mathcal{T})$, {\it i.e.\/}, subgroups closed under the topology of pointwise convergence on vertices of $\mathcal{T}$.

The final section of this note concerns $\Aut^+(\mathcal{T}_\Q)$, the closed subgroup of $\Aut(\mathcal{T}_\Q)$ generated by $\left\{\Fix(\{u,v\})\mid \{u,v\}\in E\mathcal{T}\right\}$. It is shown in~\cite[Th\'eor\`eme~4.5]{Ref1} that, if $\Aut(\mathcal{T}_\Q)$ leaves no end or proper non-empty sub-tree of $\mathcal{T}_\Q$ invariant, then $\Aut^+(\mathcal{T}_\Q)$ is a normal subgroup that is either trivial or simple. Furthermore,~\cite[\S5.7]{Ref1} shows that these hypotheses on $\Aut(\mathcal{T}_\Q)$ are satisfied if $a_{ij}\ne1$ for all $i,j\in C$. The quotient group $\Aut(\mathcal{T}_\Q)/\Aut^+(\mathcal{T}_\Q)$ is shown in~\cite[Proposition~6.7]{Ref1} to be isomorphic to the free product of a number of copies of the integers and the group of order~$2$. Typically, this quotient is infinite and, being a free product, has homomorphic images in topological groups that are not closed. We shall see that, on the other hand, homomorphic images of $\Aut^+(\mathcal{T}_\Q)$ are always closed.


\section{Decompositions for Automorphism Groups of Label-Regular Trees} \label{sec:decomp}

In what follows, a finite or infinite sequence of labels, $(c_i)_{i \in I} \subseteq C$ will be said to be \emph{compatible} with the labelling on the tree $\mathcal{T}_\Q$ if there is a path $(v_i)_{i \in I} \subseteq V\mathcal{T}_\Q$ satisfying $\lambda(v_i) = c_i$ for each $i \in I$.

\subsection{Cartan Decomposition for Label-Regular Trees}

Here we show that for any label-regular tree $\mathcal{T}_\Q$, the group $\Aut(\mathcal{T}_\Q)$ admits a $(KAK)$-, or Cartan-like, decomposition. The compact subgroup $K$ in our decomposition is the stabiliser of a vertex and the set $A$ of double coset representatives is indexed by a set of finite sequences of labels compatible with the labelling of the tree. In the case when the set $C$ of labels has just one element, $\Q$ is a $1\times1$ matrix, $\mathcal{T}_\Q$ must be regular and a compatible sequence of length $d$ is simply a string of $d$ copies of that one label. The coset representatives chosen below in Theorem~\ref{thm:kak} corresponding to these strings may then be powers of a single element, as in the usual Cartan decomposition.

\begin{thm} \label{thm:kak}
Let $\mathcal{T}_\Q$ be a label-regular tree with labels in a set, $C$, and let $K = \stab_{\Aut(\mathcal{T}_\Q)}(v)$ for a fixed vertex $v \in V\mathcal{T}_\Q$. Let $A$ be the set of all finite sequences in $C$ that are compatible with $\mathcal{T}_\Q$ and begin and end with the label $\lambda(v)$. For each $\alpha\in A$, choose $v_\alpha\in V\mathcal{T}_\Q$ and $g_\alpha\in \Aut(\mathcal{T}_\Q)$ such that the sequence of labels of vertices on the unique path from $v$ to $v_\alpha$ is $\alpha$ and $g_\alpha(v) = v_\alpha$. Then the double cosets $Kg_\alpha K$, $\alpha\in A$, are pairwise disjoint and 
\begin{equation}
\label{eq:kak}
\Aut(\mathcal{T}_\Q) = \bigsqcup_{\alpha\in A} Kg_\alpha K.
\end{equation}
\end{thm}

\begin{proof} If $g \in Kg_\alpha K$ for $\alpha \in A$, then $g(v)\in Kg_\alpha(v)$. Hence the path from $v$ to $g(v)$ is the same length as the path from $v$ to $g_\alpha(v)$ and the sequence of labels on the path from $v$ to $g(v)$ is $\alpha$. Indeed, $Kg_\alpha K$ is precisely the set of elements $g\in\Aut(\mathcal{T}_\Q)$ such that the sequence of labels on the path $v$ to $g(v)$ is equal to $\alpha$. Hence the double cosets $Kg_\alpha K$ are distinct for distinct $\alpha$, and every element of $\Aut(\mathcal{T}_\Q)$ belongs to one of these double cosets.
\end{proof}

Thus the problem of enumerating all the coset representatives in a ($KAK$)-decomposition is equivalent to enumerating all the possible distinct sequences of labels compatible with $\mathcal{T}_\Q$, starting at the vertex $v$ and extending out to a vertex of label $\lambda(v)$ at distance $d$ from $v$, for every possible distance $d$. 

Note that there is in fact a distinct $(KAK)$-decompositions of $\Aut(\mathcal{T}_\Q)$ for each label $c\in C$. The vertex $v$ in the statement of Theorem~\ref{thm:kak} is labelled $c = \lambda(v)$ and, if $\tilde{v}$ is any other vertex with the same label, there is $h\in\Aut(\mathcal{T}_\Q)$ with $h(v) = \tilde{v}$. Then the sequence of labels on the path from $\tilde{v}$ to $h(v_\alpha)$ is $\alpha$ and, replacing $v$ with $\tilde{v}$ in~\eqref{eq:kak}, the indexing set~$A$ in the disjoint union is the same, $K = \stab_{\Aut(\mathcal{T}_\Q)}(\tilde{v})$, and $\tilde{g}_\alpha := hg_\alpha h^{-1}$ are double-coset representatives.

\subsection{Bruhat Decomposition for Label-Regular Trees}

Here we show that for any label-regular tree $\mathcal{T}_\Q$, the group $\Aut(\mathcal{T}_\Q)$ admits $(BWB)$-, or Bruhat-like, decompositions. In the familiar Bruhat decomposition, $B$ is a parabolic subgroup of $G$, and $W$ is finite. When $|C|=1$ and $\mathcal{T}_\Q$ is a regular tree, $B$ may be replaced by $\Fix_{\Aut(\mathcal{T}_\Q)}(\infty)$, the fixator of an end of the tree and $W$ by the finite group $\{\iota,\sigma\}$ with $\sigma$ an inversion of $\mathcal{T}_\Q$; all fixators of ends of the regular tree are conjugate and so there is just one $B$ and one Bruhat decomposition up to conjugation. When $|C|>1$, it is no longer the case that all end fixators are conjugate and, consequently, there are many non-conjugate $(BWB)$-decompositions. It is also no longer the case that $B$ is $2$-transitive on the boundary of $\mathcal{T}_\Q$ and, consequently, there are infinitely many double cosets in each decomposition. We shall describe the $(BWB)$-decomposition for those subgroups, $B$, that contain a hyperbolic element. Observations about hyperbolic elements required are recorded in the next lemma.

\begin{lem}\label{lem:breaks}
Let $\mathcal{T}_\Q$ be a label-regular tree and suppose that $g\in \Aut(\mathcal{T}_\Q)$ is hyperbolic, {\it i.e.\/}, a translation of $\mathcal{T}_\Q$. Suppose that $\infty \in \partial \mathcal{T}_\Q$ is the attracting end of~$g$ and $-\infty$ is the repelling end. Then the bi-infinite path between $-\infty$ and $\infty$, $\ell = (v_i)_{i=-\infty}^{\infty} \subseteq V\mathcal{T}_\Q$, is the axis of $g$ and the sequence, $(\lambda(v_i))_{i=-\infty}^{\infty}$, of labels is periodic. Suppose that the period of this sequence is $p$ and that $\gamma = (c_i)_{i=1}^{p}$ is the repeated sequence of labels.  

Suppose that $\omega\in \partial \mathcal{T}_\Q\setminus\{\infty\}$ contains a ray whose sequence of labels is periodic with repeating pattern $\gamma$. Let $r = (w_i)_{i=-\infty}^{\infty} \subseteq V\mathcal{T}_\Q$ be the bi-infinite path between $\omega$ and $\infty$. Then there are integers $m \leq n \in \mathbb{Z}$ such that the sequences $(\lambda(w_{i}))_{i=-\infty}^m$ and $(\lambda(w_{i}))_{i=n}^\infty$ are periodic with repeating pattern $\gamma$ and that $n-m$ is minimised.
\end{lem}

\begin{proof}
If the distance by which $g$ translates is $p'$, then $\lambda(v_{i+p'}) = \lambda(v_i)$ for all $i\in\mathbb{Z}$ and $(\lambda(v_i))_{i=-\infty}^{\infty}$ is $p$-periodic with $p$ a divisor of $p'$. 

The path $r$ from $\omega$ to $\infty$ contains a ray belonging to $\omega$ at one end and a ray belonging to $\infty$ at the other. Hence there are $m'\leq n'\in \mathbb{Z}$ such that $(\lambda(w_{i}))_{i=-\infty}^{m'}$ and $(\lambda(w_{i}))_{i=n'}^\infty$ are periodic with repeating pattern $\gamma$. Let $m$ and $n$ be respectively the largest value of $m'$ and the smallest value of $n'$ for which this holds.
\end{proof}

The sequence of labels $\tau = (\lambda(w_{m+i}))_{i=0}^{n-m}$ in Lemma~\ref{lem:breaks} will be called the \emph{transition from $\omega$ to $\infty$} in the labelling on the path between $\omega$ and $\infty$. A transition sequence $\tau = (t_i)_{i=0}^{n-m}$ has $t_0 = c_1 = t_{n-m}$ and is such that the sequence $c_2\tau c_2$, formed by adding $c_2$ at the beginning and end of $\tau$, is compatible with the labelling of $\mathcal{T}_\Q$. It may happen that $m=n$ and the transition is a single vertex. The set of \emph{transitions for $g$}, as in Lemma~\ref{lem:breaks}, is the set of all transitions from $\omega$ to $\infty$ for $\omega\in \partial \mathcal{T}_\Q\setminus\{\infty\}$. We now describe the $(BWB)$-decomposition in the case when $B$ contains a hyperbolic element.

\begin{thm}\label{thm:bwb}
Let $\mathcal{T}_\Q$ be a label-regular tree with labels in $C$. Let $B= \Fix_{\Aut(\mathcal{T}_\Q)}(\infty)$ with $\infty\in\partial\mathcal{T}_\Q$. Suppose that $B$ contains a hyperbolic element $g$ that has $\infty$ as an attracting end. Let $W$ be the set of all transitions for $g$. Then there are elements $g_\tau\in \Aut(\mathcal{T}_{\Q})$, $\tau\in W$, such that the double cosets $Bg_\tau B$ are disjoint for distinct ~$\tau$ and
\begin{equation}
\label{eq:bwb}
G = B \sqcup \bigsqcup_{\tau\in W} Bg_\tau B.
\end{equation}
\end{thm}

\begin{proof}
For each $\tau\in W$ choose $\omega_\tau \in \partial \mathcal{T}_\Q$ such that $\tau$ is the transition from $\omega_\tau$ to~$\infty$. Then, in the notation of Lemma~\ref{lem:breaks}, the rays $(\lambda(w_{i}))_{i=-m}^{-\infty}$ and $(\lambda(w_{i}))_{i=n}^\infty$ are isomorphic. This isomorphism extends to an automorphism of $\mathcal{T}_\Q$ that maps $\infty$ to $\omega_\tau$ and which we shall call $g_\tau$. Then every $h\in g_\tau B$ maps $\infty$ to $\omega_\tau$ and, for every $h\in Bg_\tau B$, the transition from $h.\infty$ to $\infty$ is equal to $\tau$. Hence the double cosets $Bg_\tau B$ are disjoint for distinct $\tau$. 

On the other hand, consider $h\in G$. If $h \not\in B$, then $\omega = h.\infty$ satisfies the hypotheses of Lemma~\ref{lem:breaks}. Let $\tau$ be the transition from $h.\infty$ to $\infty$. Then the bi-infinite path from $h.\infty$ to $\infty$ is isomorphic to the path from $g_\tau.\infty$ to $\infty$ and this isomorphism extends to an automorphism, $b_1$ say, of $\mathcal{T}_\Q$ that fixes $\infty$ and maps $h.\infty$ to $g_\tau.\infty$. Hence $(g_\tau^{-1} b_1h).\infty = \infty$ and $g_\tau^{-1} b_1h = b_2 \in B$. Hence $h = b_1^{-1}g_\tau b_2\in \bigsqcup_{\tau\in W} Bg_\tau B$ and \eqref{eq:bwb} holds.
\end{proof}

\subsection{Enumerating Orbits in Label-Regular Trees}
The following proposition provides an explicit formula for calculating the number of distinct paths in $\mathcal{T}_\Q$ with the same sequence of labels starting from a fixed vertex $v \in V\mathcal{T}_\Q$.

\begin{prop}
\label{prop:no_of_paths}
Let $\mathcal{T}_\Q$ be a label-regular tree labelled by the elements of $C$, and $\Q = (a_{ij})_{i,j\in C}$ be a $C \times C$ matrix whose entries are non-negative integers such that $a_{ij}=0$ if and only if $a_{ji}=0$. If $\gamma = (c_i)_{i=1}^{k} \subseteq C$ is a finite sequence of labels, then given a fixed vertex $v \in V(\mathcal{T}_\Q)$ labelled $c_1$, there are exactly
$$
\mathcal{O}(\Q, \gamma) = a_{c_1 c_2} \cdot \prod_{i=2}^{k-1} ( a_{c_i c_{i+1}} - \delta_{c_{i-1} c_{i+1}})
$$
distinct paths with the labelling $(c_i)_{i=1}^{k}$ starting at the vertex $v$, where $\delta_{ij}$ denotes the Kronecker delta symbol.
\end{prop} 

\begin{proof} 
Fix a vertex $v \in \mathcal{T}_\Q$ with label $c_1$. We prove the result by induction. Clearly it is true for $k=2$, for there are exactly $a_{c_1 c_2}$ vertices labelled $c_2$ adjacent to $v$. Now assume that the result holds for all sequences of labels up to length $k$. Then given a sequence of labels $(c_1, \dots, c_k, c_{k+1}) \subseteq C$, by the induction hypothesis there are exactly $m = a_{c_1 c_2} \prod_{i=2}^{k-1} ( a_{c_i c_{i+1}} - \delta_{c_{i-1} c_{i+1}} )$ distinct paths with the labelling $(c_1, \dots, c_k)$ starting at $v$. Then, if $c_{k-1} \ne c_{k+1}$, there are exactly $ma_{c_{k} c_{k+1}}$ paths with the labelling $(c_1, \dots, c_k, c_{k+1})$ starting at $v$, and if $c_{k-1} = c_{k+1}$, there are $m(a_{c_{k} c_{k+1}} - 1)$ such paths. 
\end{proof}

Note that Proposition~\ref{prop:no_of_paths} incorporates the condition for the sequence $\gamma$ to be compatible with the labelling: if $a_{ij} = 1$ for any $i,j\in C$ and $\gamma$ has $c_k = i$ and $c_{k-1} = j = c_{k+1}$ for any~$k$, then $\mathcal{O}(\Q, \gamma) = 0$ and $\gamma$ is not compatible with the labelling. 

\begin{cor}
\label{cor:orbit_size}
Let $\mathcal{T}_\Q$ be a label-regular tree whose vertices are labelled by $C$ and fix a vertex $v \in V\mathcal{T}_\Q$. Given any vertex $u \in V\mathcal{T}_\Q$, let $\gamma = (c_i)_{i=1}^{n} \subseteq C$ be the sequence of labels on the unique path from the vertex $v$ to $u$. The size of the orbit of $u$ under the action of $\stab_{\Aut(\mathcal{T}_\Q)}(v)$ is precisely $\mathcal{O}(\Q, \gamma)$. 
\end{cor}


\section{Closedness of Range for Homomorphisms from Automorphism Groups of Label-Regular Trees} \label{sec:hom}

It is shown in this section that, for certain label-regular trees $\mathcal{T}_\Q$, every continuous homomorphism from the group $\Aut^+(\mathcal{T}_\Q)$ to a topological group has closed range. The proof uses the $(KAK)$-decomposition  established in Theorem~\ref{thm:kak} for an appropriate choice of $K$, and additional results from~\cite{Ref1} must be recalled in order to say how $K$ is chosen. First, a vertex $v$ is a \emph{point of ramification} if it has at least $3$ neighbours and, as shown in \cite[\S6.1]{Ref1}, $\Aut^+(\mathcal{T}_\Q)$ is generated by the stabilisers of the points ramification. We shall choose $K = \stab_{\Aut(\mathcal{T}_\Q)}(v)$ with $v\in V\mathcal{T}_\Q$ a point of ramification. Then $K\leq \Aut^+(\mathcal{T}_\Q)$ and, using the notation of Theorem~\ref{thm:kak} and defining $A^+ = \{\alpha\in A \mid g_\alpha\in \Aut^+(\mathcal{T}_\Q)\}$, the decomposition~\eqref{eq:kak} becomes 
\begin{equation}
\label{eq:kak2}
\Aut^+(\mathcal{T}_\Q) = \bigsqcup_{\alpha\in A^+} Kg_\alpha K.
\end{equation}

The proof also uses the following notion, which extends the idea of a contraction subgroup for a single element of a topological group, see~\cite{Ref8}, to sequences.
\begin{dfn}
\label{defn:contraction}  
Suppose that $G$ is a topological group and let $(g_i)_{i=1}^{\infty} \subseteq G$. The \emph{contraction subgroup} for $( g_i )_{i=1}^{\infty}$ is 
$$
\con( g_i ) = \left\{ x \in G : g_{i} x g_{i}^{-1} \to \id_G \text{ as } i \to \infty \right\}.
$$
\end{dfn}
\noindent  The contraction subgroup for sequences of the double-coset representatives in~\eqref{eq:kak2} is used, and it is important that it be contained in $\Aut^+(\mathcal{T}_\Q)$. The contraction subgroup is shown to be non-trivial in the proof of Theorem~\ref{thm:closed} by seeing that, for edges $\{v,w\}\in E\mathcal{T}_\Q$ with $v$ the ramification point chosen in the first paragraph, there is $x\in \con( g_i )$ that is equal to the identity on one of the half-trees obtained when $\{v,w\}$ is deleted and not the identity on the other half-tree. Such elements exist as soon as one of the entries in $\Q$ is at least~$2$, by~\cite[Proposition~5.4]{Ref1}, and belong to $\Aut^+(\mathcal{T}_\Q)$ by definition. 
\begin{thm}
\label{thm:closed}
Let $\mathcal{T}_\Q$ be a label-regular tree such that $\Aut(\mathcal{T}_\Q)$ leaves no proper non-empty subtree and no end of $\mathcal{T}_\Q$ invariant. Suppose that $\varphi: \Aut^+(\mathcal{T}_\Q) \rightarrow G$ is a continuous homomorphism to a topological group $G$. Then the range of $\varphi$ is closed in $G$.  
\end{thm}
\begin{proof}
There is nothing to prove if $\Aut^+(\mathcal{T}_\Q)$ is the trivial subgroup and so we may suppose that $\Aut^+(\mathcal{T}_\Q)$ is not trivial, in which case~\cite[Th\'eor\`eme~4.5]{Ref1} shows that it is simple. 

Consider a sequence $(g_i)_{i=1}^{\infty} \subseteq \Aut^+(\mathcal{T}_\Q)$ and suppose that $\varphi(g_i)$ converges to $\hat{g} \in G$. It must be shown that $\hat{g}\in \varphi(\Aut^+(\mathcal{T}_\Q))$. According to the decomposition of $\Aut^+(\mathcal{T}_\Q)$ into $K$-double cosets in~\eqref{eq:kak2}, there are sequences $(k_i)_{i=1}^{\infty},\ (k'_i)_{i=1}^{\infty}\subset K$ and $(\alpha_{i})_{i=1}^{\infty} \subseteq A^+$ such that $g_{i} = k_{i} g_{\alpha_i} k'_{i}$ for each~$i$. Passing to a subsequence we may suppose, by compactness of $K$, that the sequences $(k_{i})$ and $(k'_{i})$ converge to  elements $k, k' \in K$ respectively. Then 
$$
\varphi(g_{\alpha_i}) = \varphi(k_{i})^{-1} \varphi(g_{\alpha_i}) \varphi(k'_{i})^{-1} \to \varphi(k)^{-1} \hat{g} \varphi(k')^{-1}\mbox{ as }i \to \infty.
$$ 
In particular, the sequence $(\varphi(g_{\alpha_i}))_{i=1}^\infty$ converges.

If the sequence $(g_{\alpha_i})_{i=1}^{\infty}$ is bounded, it may be supposed to be constant by passing to a subsequence, so that $g_{\alpha_i} = g_\alpha$ for each $i$. Then $\hat{g} = \varphi(k) \varphi(g_\alpha) \varphi(k')$ belongs to $\varphi(\Aut^+(\mathcal{T}_\Q))$ and the proof is complete. Suppose, on the other hand, that $(g_{\alpha_i})$ is unbounded as $i\to\infty$ and set $\hat{a} := \lim_{i\to\infty} \varphi(g_{\alpha_i})$. Let $x \in \con(g_{\alpha_i})$. Then $\hat{a} \varphi(x) \hat{a}^{-1} = \lim_{i \to \infty} \varphi(g_{\alpha_i} x g_{\alpha_i}^{-1}) = \id_G$ by definition of the contraction subgroup for the sequence $(g_{\alpha_i})_{i=1}^{\infty}$ and continuity of $\varphi$. Hence the kernel of $\varphi$ contains $\con(g_{\alpha_i})$. If $\con(g_{\alpha_i})$ is not trivial, then $\varphi$ is the trivial homomorphism because $\Aut^+(\mathcal{T}_\Q)$ is simple. Therefore it suffices to complete the proof to show that, passing to a subsequence if necessary, $\con(g_{\alpha_i})$ is not trivial if $(g_{\alpha_i})$ is an unbounded sequence.

Suppose that $(g_{\alpha_i})_{i=1}^{\infty}$ is unbounded. Then we may assume, by passing to a subsequence, that for each $i\geq1$ the distance from $v$ to $g_{\alpha_i}(v)$ (see Theorem~\ref{thm:kak} for the notation) is at least $i$ and, since $\mathcal{T}_\Q$ is locally finite, that the first step of the path from $v$ to $g_{\alpha_n}(v)$ always passes through the same vertex, $w$ say. Denote the two semi-trees formed by removing the edge $\{ v,w \} \in E\mathcal{T}_\Q$ by $\mathcal{T}_v$ and $\mathcal{T}_w$, where $\mathcal{T}_v$ is the semi-tree containing the vertex $v$ and similarly for $\mathcal{T}_w$.

If infinitely many of the $g_{\alpha_i}$ are translations with $v$ on their axis, then it may be assumed that all are by passing to a subsequence. Then: $w$ is on the axis of $g_{\alpha_i}$ too; $g_{\alpha_i}^{-1}(v)\in \mathcal{T}_v$; and the distance from $g_{\alpha_i}^{-1}(v)$ to $v$ is at least $i$. Choose $x\in\Aut(\mathcal{T}_\Q)$ that fixes the semi-tree $\mathcal{T}_v$ and acts non-trivially on $\mathcal{T}_w$. Then $g_{\alpha_i} x g_{\alpha_i}^{-1}$ fixes the ball of radius $i$ around $v$. Such an $x$ exists because $\Aut(\mathcal{T}_{\Q})$ has Tits' Property~(P), see \cite[\S4.2]{Ref1}, and because \cite[Lemma~4.1]{Ref1} implies that there is a vertex in $\mathcal{T}_w$ with label $i$ and such that $a_{ij}\geq2$ for some~$j$. Hence $g_{\alpha_i} x g_{\alpha_i}^{-1} \to \id$ as $i \to \infty$ and $x \in \con(g_{\alpha_i})\setminus\{\id\}$.

If, on the other hand, only finitely many of the $g_{\alpha_i}$ are translations with $v$ on their axis, then it may be assumed that no $g_{\alpha_i}$ is a translation with $v$ on its axis. Then: $w$ is closer than $v$ to the fixed points of $g_{\alpha_i}$, if $g_{\alpha_i}$ is elliptic, or the axis of $g_{\alpha_i}$, if it is a translation; $g_{\alpha_i}^{-1}(v)\in \mathcal{T}_w$; and the distance from $g_{\alpha_i}^{-1}(v)$ to $v$ is at least $i$. Choose $x\in \Aut(\mathcal{T}_\Q)$ that fixes $\mathcal{T}_w$ and acts non-trivially on $\mathcal{T}_v$. Then the automorphism $g_{\alpha_i} x g_{\alpha_i}^{-1}$ fixes the ball of radius $i$ around $v$. Hence, $g_{\alpha_i} x g_{\alpha_i}^{-1} \to \id$ as $i\to \infty$ and $x \in \con(g_{\alpha_i})\setminus\{\id\}$.
\end{proof}

The proof of the preceding theorem showed that the following holds:

\begin{prop}
Let $\mathcal{T}_\Q$ be a label-regular tree and write $\Aut(\mathcal{T}_\Q) = \bigsqcup_{\alpha\in A} Kg_\alpha K$ as in Theorem \ref{thm:kak}. Given any sequence $({\alpha_i})_{i=1}^{\infty} \subseteq A$, either $(\alpha_i)_{i=1}^{\infty}$ is bounded, or $(g_{\alpha_i})_{i=1}^{\infty}$ has a subsequence whose contraction group is non-trivial.
\end{prop}


\bibliographystyle{amsalpha}

\end{document}